\documentclass[a4paper, reqno]{amsart}
\usepackage{hyperref,url}
\usepackage{a4wide}
\usepackage[final]{pdfpages}                
\usepackage{amsthm, amssymb, amsmath, latexsym, upgreek,accents}
\usepackage{amsfonts, mathrsfs, color}
\usepackage{graphicx,soul}

\usepackage{tikz}

\usepackage{booktabs}
\usepackage{textcomp}
\usepackage{subfigure}

\numberwithin{equation}{section}

\theoremstyle{plain}                     
\begingroup
\newtheorem{theorem}{Theorem}[section]

\newtheorem{lemma}[theorem]{Lemma}      
\endgroup
      
\theoremstyle{definition}                
\begingroup

\newtheorem{remark}[theorem]{Remark}
\endgroup

\newcommand{\R}{\mathbb R}
\newcommand{\N}{\mathbb N}

\newcommand{\C}{\mathbb C}
\newcommand{\e}{\varepsilon}

\newcommand{\semi}{m_1}
\newcommand{\cirlaw}{m_0}

\DeclareMathOperator\supp{supp}

\newcommand{\mres}{\mathbin{\vrule height 1.6ex depth 0pt width
0.13ex\vrule height 0.13ex depth 0pt width 1.3ex}}

\def\weakto{\rightharpoonup}

\title[The equilibrium measure for a nonlocal dislocation energy]{The equilibrium measure\\ 
for a nonlocal dislocation energy}

\author[M.G. Mora]{Maria Giovanna Mora} \author[L. Rondi]{Luca Rondi} \author[L. Scardia]{Lucia Scardia}

\address[M.G. Mora]{Dipartimento di Matematica, Universit\`a di Pavia, Italy}
\email{mariagiovanna.mora@unipv.it}

\address[L. Rondi]{Dipartimento di Matematica e Geoscienze, Universit\`a di Trieste, Italy}
\email{rondi@units.it}

\address[L. Scardia]{Department of Mathematical Sciences, University of Bath, United Kingdom} 
\email{L.Scardia@bath.ac.uk}

\begin{document}

\begin{abstract} 
In this paper we characterise the equilibrium measure for a nonlocal and \emph{anisotropic} weighted energy describing the interaction of positive 
dislocations in the plane. We prove that the minimum value of the energy is attained by a measure supported on the vertical axis and distributed according 
to the \emph{semi-circle} law, a well-known measure which also arises as the minimiser of purely logarithmic interactions in one dimension.
In this way we give a positive answer to the conjecture that positive dislocations tend to form \emph{vertical walls}.
This result is one of the few examples where the minimiser of a nonlocal energy is explicitly computed and the only one in the case of anisotropic kernels.
\end{abstract}

\maketitle

\begin{section}{Introduction}
In this paper we find explicitly the unique minimiser of the nonlocal energy 
\begin{equation}\label{ce}
I(\mu) = \iint_{\mathbb{R}^2\times \mathbb{R}^2} V(x-y) \,d\mu(x) \,d\mu(y) + \int_{\mathbb{R}^2} |x|^2 \,d\mu(x)
\end{equation}
defined on probability measures $\mu\in \mathcal{P}(\R^2)$, where $V$ is the interaction potential given by 
\begin{equation}\label{V:int0}
V(x)= -\log|x| + \frac{x_1^2}{|x|^2}, \qquad x=(x_1,x_2),
\end{equation}
and the second term in the energy acts as a confinement for the measure. For the precise definition of the energy $I$ see the beginning of Section~\ref{sec:2}.

The energy \eqref{ce} arises as the $\Gamma$-limit of the discrete interaction energy of a system of $n$ positive edge dislocations with Burgers vector $\mathbf{e_1}$, as $n$ tends to infinity. More precisely, $I$ is the $\Gamma$-limit of $w_n/n^2$, where
\begin{equation}\label{de}
w_n(x^1,\dots,x^n) = \sum_{i\neq j} V(x^i-x^j) + n\sum_i |x^i|^2, \qquad \{x^i\}\subset \mathbb R^2,
\end{equation}
with respect to the weak$^*$ convergence of the empirical measures $\frac1n\sum_i\delta_{x^i}$ (see Section~\ref{rmk:gamma}). 
Therefore, $I$ is the leading order or \textit{mean-field} behaviour of the Hamiltonian $w_n$, and the minimisers of $I$ represent the mean-field description of the minimisers of $w_n$, namely the equilibrium dislocation patterns at the mesoscale.  Although such minimisers have not been characterised analytically so far - neither in the discrete nor in the continuum case - they are conjectured to be vertical wall-like structures (see, e.g., \cite{DUSI, HL, LBN93}). This belief has triggered a considerable interest in dislocation walls in the engineering and mathematical literature, and interactions, upscaled behaviour and dynamics of walls have been thoroughly analysed (see, e.g., \cite{BaskMes10, Berdi, GPPS1, Hall11, vMM, SB93}). 

In this paper we give a positive answer to the conjecture. We prove that the minimiser of $I$ exists, is unique, and is given by a \textit{one-dimensional, vertical} measure, namely the semi-circle law on the vertical axis
$$
\semi:= \frac{1}{\pi}\delta_0\otimes \sqrt{2-x_2^2} \, \mathcal{H}^1\mres(-\sqrt2,\sqrt2).
$$

This is the first example of an anisotropic kernel for which the minimiser can be explicitly computed. Even in the radially symmetric case,
the explicit characterisation of the equilibrium measure has been done only for the Coulomb potential in any dimension and for the logarithmic potential in dimension one. 

In two dimensions the Coulomb potential, namely $V = -\log|\cdot|$, arises in a variety of contexts, such as, e.g., Fekete sets, orthogonal polynomials, random matrices, Ginzburg-Landau vortices, Coulomb gases.
For the same confinement term as in \eqref{ce}, the minimiser is given by the circle law $\cirlaw:= \frac{1}{\pi} \chi_{B_1(0)}$ (see, e.g., \cite{Fro, SaTo}, and the references therein). 
Although the radial component of the potential in \eqref{V:int0} is exactly the Coulomb kernel, 
the presence of the additional anisotropic term has a dramatic effect on the structure of the 
equilibrium measure. Unlike $\cirlaw$, the support of $\semi$ is one-dimensional and its density is not constant.

For the logarithmic potential in one dimension, corresponding to the so-called Log-gases energy (see, e.g., \cite{Meh, SSCG1d}), 
Wigner proved in \cite{Wi} that the semi-circle law is the unique minimiser. 
We note that the functional $I$ in \eqref{ce} coincides with the Log-gases energy on measures with support on the vertical axis, since
the anisotropic term vanishes on those measures.
Therefore if one could prove that the minimiser of $I$ is supported on the vertical axis, then the minimality of the semi-circle law would follow directly.

This is however not the strategy we use in this paper. Our approach consists of two steps: We first prove the strict convexity of $I$ on the class of measures with compact support and finite interaction energy. Strict convexity implies uniqueness of the minimiser and the equivalence between minimality and the Euler-Lagrange conditions for $I$. As a second step, we show that the semi-circle law satisfies the Euler-Lagrange conditions and hence is the unique minimiser of $I$.

For the proof of these two steps we could not rely on the machinery developed in the classical case of purely logarithmic potentials with external fields (see \cite{SaTo}), which is heavily based on  $-\log|\cdot|$ being radially
symmetric, and on it being the fundamental solution of the Laplace operator, since $V$ is neither. Similarly, although nonlocal energies are widely used and studied in the mathematical community, and the existence of their ground states and their qualitative properties have received great attention in recent years (see, e.g., \cite{BCLR, CCP, CCV, CHVY, SST}), the potential is typically required to be radially symmetric, or the singularity to be non-critical, so $V$ is not covered by their analysis.

\begin{subsection}{Our approach and main results.} 
Existence of minimisers of $I$ is straightforward, as well as the fact that minimisers have compact support. 
Our first result is the strict convexity of $I$, which entails uniqueness. As in the case of purely logarithmic interactions, strict convexity is a consequence of the following key result (see Remark~\ref{rmk:uniq}).

\begin{theorem}\label{lemma:cx}
Let $\mu_0, \mu_1 \in \mathcal{P}(\R^2)$ be measures with compact support and finite interaction energy, that is, $\int_{\R^2} (V\ast\mu_i) \, d\mu_i < + \infty$ for $i=0,1$. Then 
\begin{equation}\label{cvx0}
\int_{\R^2} V\ast (\mu_1-\mu_0) \,d(\mu_1-\mu_0) \geq 0,
\end{equation}
and the integral above is zero if and only if $\mu_0=\mu_1$.
\end{theorem}

For purely logarithmic interactions, the proof of the analogous result to Theorem~\ref{lemma:cx} (see \cite[Lemma~1.8]{SaTo}) relies on ingeniously rewriting the logarithm, according to the following formula: 
\begin{equation}\label{log:square}
-\log|x-y| = \frac1{2\pi} \int_{|z|\leq R} \frac{1}{|z-x|\,|z-y|} \, dz + \text{const.} - \log R + O\left(\frac1R\right),
\end{equation}
for $R$ sufficiently large. This trick allows one to rewrite the nonlocal term by `unfolding' the convolution, and transforming it into the integral of an exact square, which immediately implies the non-negativity of the integral. 

It is not clear whether a similar rewriting as in \eqref{log:square} is valid for the potential $V$, so we use a different approach. This is based on the intuition that, if we could rewrite the convolution in \eqref{cvx0} in Fourier space, then heuristically we would have that 
\begin{equation}\label{unfolding:V}
\int_{\R^2} V\ast (\mu_1-\mu_0) \,d(\mu_1-\mu_0) = \int_{\R^2} \hat V |\hat{\mu}_1-\hat{\mu}_0|^2 d\xi,
\end{equation}
and hence proving that $\hat{V}>0$ would imply the theorem.

As a first step, then, we compute the Fourier transform of $V$ (see Lemma~\ref{FTV}), which is a tempered distribution.
Unfortunately, the Fourier transform $\hat V$ is not a positive distribution (see Remark~\ref{r:transV0}), but we can show that $\hat{V}>0$ for positive test functions that are zero at $\xi=0$. The key remark is that this is enough to conclude, since 
$\mu_1-\mu_0$ is a neutral measure and thus 
the test function $|\hat{\mu}_1-\hat{\mu}_0|^2$ in \eqref{unfolding:V} is zero at $\xi=0$. This heuristic argument can in fact be made rigorous, and this is the heart of the proof of Theorem~\ref{lemma:cx}.

\

The explicit determination of the minimiser of $I$ is the main result of this paper.

\begin{theorem}\label{thm:chara}
The measure 
\begin{equation}\label{eq:sc-intro}
\semi = \frac{1}{\pi}\delta_0\otimes \sqrt{2-x_2^2} \, \mathcal{H}^1\mres(-\sqrt2,\sqrt2)
\end{equation}
satisfies the conditions
\begin{align}
&(V\ast \semi)(x) + \frac{|x|^2}2 = \frac12 +\frac12\log2 \qquad  \text{for every } x\in\supp \semi,
\label{EL-2-intro}
\\
&(V\ast \semi)(x) + \frac{|x|^2}2 \geq  \frac12 +\frac12\log2 \qquad \text{for every }x\in \R^2,
\label{EL-1-intro}
\end{align}
and hence is the unique minimiser of $I$.
\end{theorem}

The proof of Theorem~\ref{thm:chara} consists of two parts: In the first part we show that \eqref{EL-2-intro}--\eqref{EL-1-intro}
are the Euler-Lagrange conditions for $I$ relative to $\semi$, and that the Euler-Lagrange conditions uniquely characterise the minimiser of $I$. This is standard and can be done as in the purely logarithmic case. In the second part of the proof we show that
$\semi$ satisfies \eqref{EL-2-intro}--\eqref{EL-1-intro}. Since on $\supp \semi$ the potential $V$ reduces to the logarithm in one dimension, \eqref{EL-2-intro} follows from the minimality of the semi-circle law for the Log-gases energy.
Proving that $\semi$ satisfies also \eqref{EL-1-intro} is instead original and extremely challenging. 

We note that one of the two Euler-Lagrange conditions must fail for any measure other than the minimiser. This suggests that in order to prove \eqref{EL-1-intro} we need to estimate the function $V\ast \semi$ in $\R^2$ with great precision and accuracy. We achieve this by computing the derivative of $V\ast \semi$ with respect to $x_1$ \textit{exactly}, and by 
showing that, on account of \eqref{EL-2-intro} and a symmetry argument, \eqref{EL-1-intro} can be reduced to proving that 
the derivative with respect to $x_1$ of 
$$
F(x):=(V\ast \semi)(x) + \frac{|x|^2}{2}
$$
is positive in the first quadrant. This is in turn equivalent to the claim 
\begin{equation}\label{mira}
\Re\left(z\, \partial_z g(z)\right)  > 0 \qquad  \text{for every }z\in \C \text{ with }\Re z>0,\ \Im z >0,
\end{equation}
where 
$$
g(z) = \frac1{2\pi} \int_{-\pi}^{\pi}\log |z-\cos\theta| \,d\theta,
$$
$\partial_z$ denotes the complex derivative, and $\Re$, $\Im$ denote the real and imaginary part. 
By applying the Joukowsky transformation in the complex plane (see, e.g., \cite[Example~1.3.5]{SaTo}) the integral in the definition of $g$ can be 
explicitly computed, so that the claim \eqref{mira} can be checked directly.
\end{subsection}

\begin{subsection}{Discussion}
The research of this paper was driven by several aims. To start with, we wanted to investigate the minimality of dislocation walls, conjectured in the literature, by means of a solid mathematical approach. Secondly, we wanted to push the methods developed for nonlocal energies beyond the case of radially symmetric potentials, still retaining the critical, logarithmic singularity at zero. Finally, we wanted to explore the connection between the theory of vortices and the theory of dislocations, which has been successfully exploited so far in the case of discrete and screw dislocations (see, e.g., \cite{ACP11, AP14}).

The literature on nonlocal interaction energies is vast. Under the assumption that
the interaction potential is \emph{radially symmetric}, several authors have investigated qualitative properties of energy minimisers, from existence and uniqueness of the equilibrium measure \cite{CCV, CHVY}  to its  confinement \cite{CCP, CdFFLS} and to the regularity of the density of the minimisers \cite{CDM}. In all the aforementioned results radial symmetry is a crucial assumption and minimisers are radially symmetric.
This assumption is relaxed in \cite{BCLR}, where the authors face the interesting question of estimating the dimension of the support of minimisers
in terms of the singularity of the potential at zero.  
However, this result requires the singularity to be subcritical, which is not the case for the potential in \eqref{V:int0}, and 
only provides a lower bound on the dimension.

Moreover, in our case we have an explicit potential coming from dislocation theory, and we find the equilibrium measure \emph{explicitly}.
In this respect, our paper is more closely related to classical, Frostman-type results on existence, uniqueness and characterisation of the extremal measure for weighted energies. As in the classical case, we consider a radially symmetric and convex weight, which corresponds to the external field $|x|^2$ in the confinement term in \eqref{ce}. The additional anisotropic term in the potential $V$, however, makes our analysis substantially different from the case of a purely logarithmic interaction. 

\subsubsection{Extensions and open questions}
Various extensions of the present work would be interesting. The type of weight, or external field, in the energy \eqref{ce} is chosen for convenience; we plan to consider other types of fields, in analogy with the classical logarithmic case. From the mechanical point of view, this would correspond to testing the stability of vertical-wall structures under different loadings.

In particular, in the absence of an external field, one could ask the question of finding the extremal measure in the class of probability measures supported on a given set $E\subset \R^2$. For purely logarithmic interactions the extremal measure is supported on the boundary of $E$; it would be interesting to see whether the anisotropic term would still force a vertical support of the equilibrium measure as in the case treated in this paper.

The case of signed measures, corresponding to the presence of both positive and negative dislocations, is our long-term goal. The minimising arrangements for the discrete energy are conjectured to be \emph{Taylor lattices}, namely structures where vertical walls of positive dislocations are alternated with vertical walls of negative dislocations, but with a relative vertical shift. The mathematical treatment of discrete systems of positive and negative dislocations, as well as their limit behaviour for a large number of dislocations, are however still at a preliminary stage.

Finally, the results in this paper raise the intriguing question of understanding the effect of the anisotropy on the dimension of the support of minimisers.
We plan to investigate this issue for more general interaction potentials.

\end{subsection}

\begin{subsection}{Plan of the paper}
In Section~\ref{sec:2} we discuss existence and uniqueness of the minimiser of~$I$ and we prove Theorem~\ref{lemma:cx}. 
The derivation of the Euler-Lagrange conditions and the proof of Theorem~\ref{thm:chara} are the subject of Section~\ref{sec:3}.

\end{subsection}

\end{section}


\begin{section}{Existence and uniqueness of the minimiser of $I$}\label{sec:2}

In this section we prove existence and uniqueness of the minimiser of the nonlocal energy $I$ in \eqref{ce}.
We start by providing the precise definition of $I$.
The interaction potential $V$ is defined by
\begin{equation}\label{V:int}
V(x) := -\log|x| + \frac{x_1^2}{|x|^2}
\end{equation}
for $x=(x_1,x_2)\in\R^2$, $x\neq0$, and is extended to $x=0$ by continuity,
that is, $V(0):=+\infty$.
Let $\mathcal{P}(\mathbb{R}^2)$ denote 
the class of all positive Borel measures on $\R^2$ with unitary mass.
For every $\mu\in \mathcal{P}(\mathbb{R}^2)$ we define
\begin{equation}\label{E:wdef}
I(\mu) := \iint_{\R^2\times \R^2} \left(V(x-y) + \frac12(|x|^2+|y|^2)\right) d\mu(x)\, d\mu(y).
\end{equation}
It is immediate to see that the integrand is non-negative and bounded from below; indeed, 
we have
\begin{eqnarray}
V(x-y) + \frac12(|x|^2+|y|^2) & \geq &
-\log|x-y| + \frac{1}{e} (|x|^2+|y|^2) +  \left(\frac12-\frac{1}{e}\right)(|x|^2+|y|^2) \nonumber\\
&\geq & -\log|x-y| + \frac{1}{2e} |x-y|^2 + \left(\frac12-\frac{1}{e}\right)\, (|x|^2+|y|^2) \nonumber\\
&\geq & \left(\frac12-\frac{1}{e}\right)\, (|x|^2+|y|^2). \label{bound:compact}
\end{eqnarray}
Therefore, the energy \eqref{E:wdef} is well defined on positive measures $\mu \in \mathcal{P}(\mathbb{R}^2)$, possibly equal to $+\infty$.
The representation \eqref{ce} of $I$ coincides with \eqref{E:wdef} whenever the first integral in \eqref{ce} is not $-\infty$.

\begin{subsection}{Existence of a minimiser of $I$}
By inequality \eqref{bound:compact} we deduce that
\begin{equation}\label{bound:2}
I(\mu) \geq \left(1-\frac{2}{e}\right) \int_{\R^2}|x|^2\,d\mu(x).
\end{equation}
This implies that $\inf I \geq 0$.

It is easy to see that if $\cirlaw= \frac{1}{\pi} \chi_{B_1(0)}$, then $I(\cirlaw)<+\infty$. Therefore, $\inf I <+\infty$.

Let now $(\mu_n)\subset \mathcal{P}(\R^2)$ be a minimising sequence for $I$. By the bound \eqref{bound:2} we 
deduce that the sequence $(\mu_n)$ is tight and therefore converges narrowly, up to a subsequence, to some $\mu\in \mathcal{P}(\R^2)$.
Since the functional $I$ is lower semicontinuous with respect to narrow convergence, the existence of a minimiser follows immediately. 

\end{subsection}

\begin{subsection}{Minimisers of $I$ have compact support.}\label{sec:cpt}
Let $\mu$ be a minimiser of $I$; in particular, $I(\mu)<+\infty$. By \eqref{bound:compact} we have that there exists a compact set $K\subset \R^2$ such that $\mu(K)>0$ and
\begin{equation}\label{growth:V}
V(x-y) + \frac12(|x|^2+|y|^2) > I(\mu) +1 \qquad \text{outside } K\times K.
\end{equation}
We now show that $\supp \mu\subset K$. If not, then $\mu(K)<1$; but in this case one can easily prove that the measure $\tilde \mu \in \mathcal{P}(\R^2)$ defined as 
$$
\tilde\mu:= \frac{\mu\mres K}{\mu(K)}
$$
has lower energy, against the minimality of $\mu$. Indeed, by \eqref{growth:V}, we have
\begin{eqnarray*}
I(\mu) &= &\iint_{K\times K}\left(V(x-y) + \frac12(|x|^2+|y|^2)\right) d\mu(x)\, d\mu(y) 
\\
&& {}+ \iint_{(K\times K)^c} \left(V(x-y) + \frac12(|x|^2+|y|^2)\right) d\mu(x)\, d\mu(y) \\
&> & (\mu(K))^2 I(\tilde\mu) + (1-(\mu(K))^2)(I(\mu) +1),
\end{eqnarray*}
and hence 
$$
I(\tilde\mu) < I(\mu) - \frac{1- (\mu(K))^2}{(\mu(K))^2} < I(\mu),
$$
which contradicts the minimality.

\end{subsection}

\begin{subsection}{Uniqueness of the minimiser of $I$.} As in the case of purely logarithmic interactions, uniqueness follows by the strict convexity of the energy on probability measures with compact support and finite interaction energy. 
Our proof of the strict convexity of the energy is however completely different and new, and is based on the computation of the  Fourier transform of the potential $V$, which we show to be strictly positive outside the origin.

As a preliminary step, we compute the Fourier transform of the potential $V$.

\begin{lemma}[Fourier transform of $V$]\label{FTV}
The Fourier transform $\hat V$ of $V$ is the tempered distribution given by 
\begin{equation}\label{hatV}
\langle \hat V, \varphi\rangle
 = \Big(\frac{1}{2}+\gamma+\log\pi \Big) \varphi(0)
+ \frac1\pi \int_{|\xi|\leq 1}(\varphi(\xi)-\varphi(0))\frac{\xi_2^2}{|\xi|^4}\, d\xi
+ \frac1\pi \int_{|\xi|>1}\varphi(\xi)\frac{\xi_2^2}{|\xi|^4} \,d\xi
\end{equation}
for every $\varphi\in \mathcal S$, where $\mathcal S$ denotes the Schwartz space and $\gamma$ is the Euler constant. 
\end{lemma}

\begin{proof}
Since $V\in L^1_{\mathrm{loc}}(\R^2)$ and has a logarithmic growth at infinity, we have that $V\in{\mathcal S}'$,
hence $\hat V\in{\mathcal S}'$. We recall that $\hat V$ is defined by the formula
$$
\langle \hat V, \varphi\rangle := \langle V, \hat \varphi\rangle \qquad \text{for every } \varphi\in{\mathcal S}
$$
where, for $\xi\in\R^2$,
\begin{equation}\label{hatphi}
\hat \varphi(\xi):=\int_{\R^2}\varphi(x)e^{-2\pi i\xi\cdot x}\, dx.
\end{equation}
It is convenient to rewrite $V$ as
$$
V(x)=- \log |x| +\frac12 +\frac12 \frac{x_1^2-x_2^2}{|x|^2}.
$$
If we consider the rotation $R:\R^2\to\R^2$ defined by
$$
R(x)=\frac{1}{\sqrt2}(x_1-x_2,x_1+x_2)
$$ 
for every $x\in\R^2$, then we obtain
$$
V(x)=2\pi F(R(x)),
$$
where
$$
F(y)=-\frac{1}{2\pi} \log |y| + \frac{1}{4\pi} + \frac{1}{2\pi}\frac{y_1y_2}{|y|^2}.
$$
Moreover, since $R$ is a rotation, we have that
\begin{equation}\label{Fourier-rot}
\hat V=2\pi(F\circ R)\,\hat{}
=2\pi\hat F\circ R.
\end{equation}
It is therefore sufficient to compute the Fourier transform of $F$. To this purpose, we observe that
$$
\frac{1}{2\pi}\frac{y_1y_2}{|y|^2}=- y_1\partial_{y_2}\Big( -\frac{1}{2\pi} \log |y|\Big), 
$$
hence, setting
$$
A(y):=-\frac{1}{2\pi} \log |y|,
$$
we have
$$
\hat F(\xi)=\hat A(\xi) +\frac{1}{4\pi} \delta_0(\xi) + \xi_2\partial_{\xi_1}\hat A(\xi),
$$
where we used that $\partial_{\xi_j}\hat f=-2\pi i(x_jf)\,\hat{}$ and $(\partial_{x_j}f)\,\hat{}=2\pi i\xi_j\hat f$.

Let now $\beta\in(0,2)$ and let $A_\beta$ be the Riesz potential (up to an additive constant) of order $\beta$,
defined by
$$
A_\beta(y):=\frac{\Gamma(1-\tfrac\beta2)}{\Gamma(\tfrac\beta2)2^\beta\pi}(|y|^{\beta-2}-1),
$$
where $\Gamma$ is the Gamma function. 
Since
\begin{equation}\label{Gamma}
\Gamma(1-\tfrac\beta2)=\frac{2\Gamma(2-\tfrac\beta2)}{2-\beta}, \qquad \Gamma(1)=1,
\end{equation}
one can show that, as $\beta\to2^-$,
$A_\beta$ converges pointwise to $A$,
thus $\hat A_\beta$ converges to $\hat A$ in the sense of tempered distributions (see, e.g., \cite[Chapter~4, page~151]{Fol}). Therefore, setting
$$
F_{\beta}(y):=A_\beta(y)+\frac{1}{4\pi}- y_1\partial_{y_2}A_\beta(y), 
$$
we deduce that $\hat F_{\beta}\to \hat F$ in the sense of tempered distributions, as $\beta\to2^-$.
It is well known that
$$
\hat A_\beta(\xi)=\frac{1}{(2\pi)^\beta|\xi|^\beta}
-\frac{\Gamma(1-\tfrac\beta2)}{\Gamma(\tfrac\beta2)2^\beta\pi}\delta_0(\xi).
$$
By straightforward computations we have
$$
\hat F_{\beta}(\xi)=\frac{1}{(2\pi)^\beta|\xi|^{\beta+2}}
(|\xi|^2-\beta\xi_1\xi_2)+\frac{1}{4\pi} \delta_0(\xi) 
-\frac{\Gamma(1-\tfrac\beta2)}{\Gamma(\tfrac\beta2)2^\beta\pi}\delta_0(\xi),
$$
where we used that $\xi_2\delta_0(\xi)=0$ in the sense of distributions.

Let now $\varphi\in\mathcal S$. Taking into account \eqref{Gamma} and the fact that
$$
\int_{|\xi|\leq 1}\frac{1}{(2\pi)^\beta|\xi|^{\beta+2}}
(|\xi|^2-\beta\xi_1\xi_2)\, d\xi
= \frac{1}{(2\pi)^{\beta-1}}\frac1{2-\beta},
$$
we compute
\begin{eqnarray}
\langle \hat F_{\beta}, \varphi\rangle
& = & \frac{1}{4\pi}\varphi(0)
+ \int_{|\xi|>1}\varphi(\xi)\frac{1}{(2\pi)^\beta|\xi|^{\beta+2}}
(|\xi|^2-\beta\xi_1\xi_2)\, d\xi
\nonumber
\\
&&
+ \int_{|\xi|\leq 1}(\varphi(\xi)-\varphi(0))\frac{1}{(2\pi)^\beta|\xi|^{\beta+2}}
(|\xi|^2-\beta\xi_1\xi_2)\, d\xi
\nonumber
\\
&&
+ \frac{1}{\Gamma(\tfrac\beta2)2^{\beta-1}\pi}\frac{\pi^{2-\beta}\Gamma(\tfrac\beta2)-\Gamma(2-\tfrac\beta2)}{2-\beta} \varphi(0).
\label{hatG}
\end{eqnarray}
An application of l'Hopital's rule shows that
$$
\lim_{\beta\to2^-}\frac{\pi^{2-\beta}\Gamma(\tfrac\beta2)-\Gamma(2-\tfrac\beta2)}{2-\beta}
=-\Gamma'(1)+\log\pi=\gamma+\log\pi,
$$
where we used that $\Gamma'(1)=-\gamma$. We can now pass to the limit in \eqref{hatG}, as $\beta\to2^-$ (note that
in the second integral on the right-hand side $|\varphi(\xi)-\varphi(0)|$ is at least or order $|\xi|$, 
as $\xi\to0$, so that integrability is guaranteed), and obtain
\begin{eqnarray*}
\langle \hat F, \varphi\rangle
& = & \frac{1}{4\pi}\varphi(0)
+ \int_{|\xi|>1}\varphi(\xi)\frac{1}{(2\pi)^2|\xi|^4}
(|\xi|^2-2\xi_1\xi_2)\, d\xi
\\
&&
+ \int_{|\xi|\leq 1}(\varphi(\xi)-\varphi(0))\frac{1}{(2\pi)^2|\xi|^4}
(|\xi|^2-2\xi_1\xi_2)\, d\xi
+ \frac{1}{2\pi}(\gamma+\log\pi) \varphi(0).
\end{eqnarray*}
By applying \eqref{Fourier-rot} we deduce \eqref{hatV}.
\end{proof}

\begin{remark}\label{r:transV0}
 By Lemma~\ref{FTV} we deduce that
\begin{equation}\label{transV0}
\langle \hat V, \varphi \rangle = \frac1{\pi}\int_{\R^2}\frac{\xi_2^2}{|\xi|^4}\varphi(\xi)\, d\xi
\end{equation}
for every $\varphi\in\mathcal S$ with $\varphi(0)=0$. Hence, $\langle \hat V, \varphi \rangle>0$ for
every $\varphi\in\mathcal S$ with $\varphi(0)=0$ and $\varphi\geq0$, $\varphi\not\equiv0$.

Note, however, that $\hat V$ is not positive on $\mathcal S$. Indeed, let us take $r_0\in(0,1)$, to be chosen later, and let us consider any radial function $\varphi\in C^\infty_c(B_{r_0}(0))$ such that $\varphi(0)>0$ and $0\leq\varphi\leq \varphi(0)$.
Following the notation of Lemma~\ref{FTV}, for any $\beta\in(0,2)$ we have
$$\langle \hat F_{\beta},\varphi\rangle\leq \varphi(0)\left(\int_{|\xi|\leq r_0}
\frac{1}{(2\pi)^\beta
|\xi|^{\beta+2}}
(|\xi|^2-\beta\xi_1\xi_2)\, d\xi +\frac{1}{4\pi} 
-\frac{\Gamma(1-\tfrac\beta2)}{\Gamma(\tfrac\beta2)2^\beta\pi}\right)=:C(r_0,\beta)\varphi(0).$$
Arguing as in the proof of Lemma~\ref{FTV}, we have that
$$\lim_{\beta\to 2^-}C(r_0,\beta)=\frac{1}{2\pi}\left(\gamma+\log(\pi r_0)+\frac 12\right).$$
It is enough to pick $r_0$ such that $\gamma+\log(\pi r_0)+\frac 12\leq -1$ to obtain that
$$\langle \hat V, \varphi\rangle \leq -\varphi(0)<0.$$
\end{remark}

We are now in a position to prove Theorem~\ref{lemma:cx}, which is the key step to 
deduce the strict convexity of the interaction energy on probability measures
with compact support (see Remark~\ref{rmk:uniq}).

\begin{proof}[Proof of Theorem~\ref{lemma:cx}]
We first prove the inequality \eqref{cvx0} for test functions in $\mathcal S$ and for functions in $L^2(\R^2)$  
with zero average, and then we extend the result to measures, by means of a careful approximation result.\smallskip

\noindent
\textit{Step~1: Inequality on test functions.}
For $\varphi\in {\mathcal S}$ we define
$$
\check \varphi(x)=\varphi(-x), \qquad \tau_x \varphi(y)=\varphi(y-x)
$$
for every $x,y\in\R^2$. Denoting by $\hat \varphi$ the Fourier transform of $\varphi$ defined in \eqref{hatphi},
we have that
\begin{equation}\label{antitrans}
 \check{\hat{\hat \varphi}}=\varphi
\end{equation}
(see, e.g., \cite[Theorem~7.7]{Rud}). Moreover, for $u\in{\mathcal S}'$ and $\varphi\in {\mathcal S}$ we define
$$
(u\ast \varphi)(x):=\langle u, \tau_x\check \varphi\rangle \qquad \text{for every } x\in\R^2.
$$
From \cite[Theorem~7.19]{Rud} it follows that $u\ast \varphi\in {\mathcal S}'$ and
\begin{equation}\label{F-conv}
(u\ast \varphi)\,\hat{} =\hat \varphi\,\hat u.
\end{equation}

We now prove that
\begin{equation}\label{inter>0}
\langle u\ast \varphi, \varphi\rangle= \langle \hat u, |\hat \varphi|^2\rangle
\end{equation}
for every $u\in{\mathcal S}'$ and $\varphi\in {\mathcal S}$. Let $\varphi\in {\mathcal S}$. We denote the conjugate of $\varphi$ in $\C$ by $\bar \varphi$.
Note that in our framework $\varphi$ is always a real-valued function, but $\hat \varphi$ may be complex-valued. 
By \eqref{antitrans} and \eqref{F-conv} we have that
\begin{equation}\label{tr-chain}
\langle u\ast \varphi, \varphi\rangle = \langle (u\ast \varphi)\,\hat{}, \check{\hat{\varphi}}\rangle =\langle \hat u, \hat \varphi\, \bar{\hat \varphi}\rangle,
\end{equation} 
where we used that $\check{\hat\varphi}=\hat{\check\varphi}$
and that
$$
 \check{\hat{\varphi}}(\xi)= \int_{\R^2} \varphi(x)e^{2\pi i\xi\cdot x}\, dx
=\bar{\hat{ \varphi}}.
$$
This proves \eqref{inter>0}.

Let now $\varphi\in {\mathcal S}$ be such that $\int_{\R^2}\varphi(x)\,dx=0$; note that $\hat \varphi(0)=\int_{\R^2}\varphi(x)\,dx=0$.
By \eqref{inter>0} applied to $u=V$, where $V$ is the interaction potential in \eqref{V:int}, we deduce that
\begin{equation}\label{strict:test}
\int_{\R^2} (V\ast \varphi) \varphi\, dx=\langle \hat V, |\hat \varphi|^2\rangle=\frac1{\pi}\int_{\R^2}\frac{\xi_2^2}{|\xi|^4}|\hat \varphi(\xi)|^2\, d\xi,
\end{equation}
where the last equality follows from Remark~\ref{r:transV0} since $\hat \varphi(0)=0$.\smallskip

\noindent
\textit{Step~2: Inequality on $L^2$ functions.}
Let $f\in L^2(\R^2)$ be such that $\int_{\R^2}f(x)\,dx=0$ and with compact support.
Let $(\varphi_k)\subset {\mathcal S}$ be a sequence converging to $f$ in $L^2(\R^2)$ with $\int_{\R^2}\varphi_k(x)\,dx=0$. In particular, $\hat\varphi_k\to\hat f$ in $L^2(\R^2)$
and $\hat\varphi_k(0)=0$ for every~$k$. Therefore, by \eqref{strict:test},
$$
\int_{\R^2} (V\ast \varphi_k) \varphi_k\, dx= \frac1{\pi}\int_{\R^2} \frac{\xi_2^2}{|\xi|^4} |\hat \varphi_k(\xi)|^2\, d\xi \qquad 
\text{for every } k.
$$
Passing to the limit as $k\to\infty$, we deduce that
\begin{equation}\label{ineqVf}
\int_{\R^2} (V\ast f) f\, dx \geq \frac1{\pi}\int_{\R^2} \frac{\xi_2^2}{|\xi|^4} |\hat f(\xi)|^2\, d\xi.
\end{equation}
Note that we can pass to the limit in the interaction energy since $V\in L^1_{\mathrm{loc}}(\R^2)$ and 
we can assume the supports of $\varphi_k$ and $f$ to be uniformly bounded. In the right-hand side we used Fatou's lemma. Thus, we have proved that \eqref{ineqVf} holds for every $f\in L^2(\R^2)$ with $\int_{\R^2}f(x)\,dx=0$ and compact support.\smallskip

\noindent
\textit{Step~3: Inequality on measures.}
Let $\mu_0, \mu_1 \in \mathcal{P}(\R^2)$ be as in the statement of the theorem, that is, such that 
\begin{equation}\label{finite12}
\int_{\R^2} (V\ast \mu_0)\, d\mu_0<+\infty, \qquad
\int_{\R^2} (V\ast \mu_1)\, d\mu_1<+\infty
\end{equation}
and with compact support. Let $\nu:=\mu_1-\mu_0$.
Note that $\nu$ is a bounded measure with compact support and $\int_{\R^2}\, d\nu=0$.

Assume now that there exist $(\mu_0^h)$, $(\mu_1^h)$ in $L^2(\R^2)$ with uniformly bounded compact supports
such that $\mu_i^h\geq0$,
\begin{equation}\label{muih0}
\int_{\R^2}\mu_i^h(x)\, dx=1,
\end{equation}
\begin{equation}\label{muih1}
\mu_i^h\weakto\mu_i \qquad \text{ narrowly, as } h\to0,
\end{equation}
and
\begin{equation}\label{muih2}
\lim_{h\to0}\int_{\R^2} (V\ast \mu^h_i)\mu^h_i\, dx = \int_{\R^2} (V\ast \mu_i)\, d\mu_i
\end{equation}
for $i=0,1$. We postpone the proof of \eqref{muih0}--\eqref{muih2} to Step~4. Set $\nu^h:=\mu^h_1-\mu^h_0$. Since $\nu^h\in L^2(\R^2)$, $\int_{\R^2}\nu^h(x)\,dx=0$, and $\nu^h$ has compact support, we can apply \eqref{ineqVf} to $\nu^h$ for every~$h$. We obtain
$$
\int_{\R^2} (V\ast \nu^h) \nu^h\, dx \geq \frac1{\pi}\int_{\R^2} \frac{\xi_2^2}{|\xi|^4} |\hat \nu^h(\xi)|^2\, d\xi \qquad
\text{for every }h.
$$
Therefore,
$$
\limsup_{h\to0}\int_{\R^2} (V\ast \nu^h) \nu^h\, dx \geq 
\liminf_{h\to0}\frac1{\pi}\int_{\R^2} \frac{\xi_2^2}{|\xi|^4} |\hat \nu^h(\xi)|^2\, d\xi
\geq \frac1{\pi}\int_{\R^2} \frac{\xi_2^2}{|\xi|^4} |\hat \nu(\xi)|^2\, d\xi,
$$
where the last inequality follows from Fatou's lemma and the fact that $\hat\nu^h\to\hat\nu$ pointwise (note in particular that, since $\nu$ is a bounded measure, $\hat \nu$ is continuous).
We now look at the left-hand side. We have
\begin{equation}\label{decomp-nu}
\int_{\R^2} (V\ast \nu^h) \nu^h\, dx
= \int_{\R^2} (V\ast \mu_1^h) \mu_1^h\, dx + \int_{\R^2} (V\ast \mu_0^h) \mu_0^h\, dx
- 2 \int_{\R^2} (V\ast \mu_1^h) \mu_0^h\, dx.
\end{equation}
The convergence of the first two integrals at the right-hand side is guaranteed by \eqref{muih2}. As for the last integral, we write
$$
\limsup_{h\to0} - 2 \int_{\R^2} (V\ast \mu_1^h) \mu_0^h\, dx
=- 2\liminf_{h\to0} \int_{\R^2} (V\ast \mu_1^h) \mu_0^h\, dx
\leq - 2 \int_{\R^2} (V\ast \mu_1) \, d\mu_0,
$$
where the last inequality follows from \eqref{muih1} by lower semicontinuity. Indeed,
$V$ is continuous and bounded from below on the uniformly bounded supports of $\mu_i^h$; thus, 
the lower semicontinuity of $\int (V\ast \mu_1)\, d\mu_0$ can be easily proved by considering truncations of $V$ from above.

Combining the previous equations together, we conclude that
\begin{equation}\label{crucial-ineq}
 \int_{\R^2} (V\ast \nu)\, d\nu\geq \frac1{\pi}\int_{\R^2} \frac{\xi_2^2}{|\xi|^4} |\hat \nu(\xi)|^2\, d\xi.
\end{equation}

If the left-hand side of \eqref{crucial-ineq} is equal to $0$,
then 
$$
\frac{\xi_2^2}{|\xi|^4} |\hat \nu(\xi)|^2=0 \qquad \text{for a.e.\ } \xi\in\R^2.
$$
Therefore, $\hat \nu(\xi)=0$ for a.e.\ $\xi$ with $\xi_2\neq0$. By continuity of $\hat\nu$ this implies
$\hat \nu(\xi)=0$ for every $\xi\in\R^2$. Thus, $\nu=0$, hence $\mu_0=\mu_1$.

We have therefore proved the thesis of the theorem.\smallskip

\noindent
\textit{Step~4: Approximation result.}
To prove \eqref{muih0}--\eqref{muih2} we proceed as in \cite[Theorem~3.3]{MPS}, Step~1 in the proof of the limsup inequality. We apply the approximation procedure described there to $\mu_0$ and $\mu_1$, separately; the $\mu^h_i$ defined in this way, for $i=0,1$, are in $L^2(\R^2)$, are non-negative, have uniformly bounded supports, and satisfy
\eqref{muih0} and \eqref{muih1}. 

To prove \eqref{muih2}, we argue as follows.
For $M>0$ we consider the truncated function $V_M:=V\wedge M$ and we write
$V=V_M + (V- V_M)$.
The function $V_M$ is bounded on bounded sets and continuous. Since the supports of $\mu^h_i$
are uniformly bounded, narrow convergence \eqref{muih1} yields
$$
\lim_{h\to0}\int_{\R^2} (V_M\ast \mu^h_i)\mu^h_i\, dx  =  \int_{\R^2} (V_M\ast \mu_i)\, d\mu_i
\leq \int_{\R^2} (V\ast \mu_i)\, d\mu_i
$$
for $i=0,1$. Therefore, \eqref{muih2} is proved if we show that
\begin{equation}\label{claim-oldp}
\lim_{M\to\infty}\limsup_{h\to0} 
\int_{\R^2}((V-V_M)\ast \mu^h_i)\mu^h_i\, dx=0 
\end{equation}
for $i=0,1$. Note also that we can replace $\R^2$ with a bounded domain in the integral above (in \cite[Theorem~3.3]{MPS} the integrals are on a bounded set $\Omega$ and not on $\R^2$) since the measures have uniformly bounded supports.
Since
\begin{equation}\label{Vbounds}
-\log |x|\leq V(x)\leq 1-\log |x|
\end{equation}
for every $x\neq0$, claim \eqref{claim-oldp} can be proved by repeating the argument in \cite[Theorem~3.3]{MPS} verbatim.
\end{proof}

\begin{remark}[Strict convexity and uniqueness of the minimiser of $I$]\label{rmk:uniq}
Theorem~\ref{lemma:cx} implies the strict convexity of $I$ on the class of probability measures with compact support and finite 
interaction energy. Indeed, let $\mu_0,\mu_1\in \mathcal{P}(\R^2)$ be two measures with compact support and finite
interaction energy such that $\mu_0\neq\mu_1$.
Inequality \eqref{cvx0} implies that
\begin{equation}\label{cvx0-mg}
2\int_{\R^2} V\ast \mu_1 \,d\mu_0 < \int_{\R^2} V\ast \mu_0 \,d\mu_0+
\int_{\R^2} V\ast \mu_1 \,d\mu_1.
\end{equation}
For any $t\in(0,1)$, set
$\mu_t:=t\mu_1+(1-t)\mu_0$
and compute
$$\int_{\R^2} V\ast \mu_t\,d\mu_t =
t^2\int_{\R^2} V\ast \mu_1 \,d\mu_1+2t(1-t)\int_{\R^2} V\ast \mu_1 \,d\mu_0+(1-t)^2
\int_{\R^2} V\ast \mu_0 \,d\mu_0.$$
Using \eqref{cvx0-mg}, we immediately infer that
$$
\int_{\R^2} V\ast \mu_t\,d\mu_t
< t \int_{\R^2} V\ast \mu_1 \,d\mu_1+(1-t) \int_{\R^2} V\ast \mu_0 \,d\mu_0.
$$

Since minimisers of $I$ have compact support and finite 
interaction energy, the property above implies uniqueness of the minimiser.
\end{remark}

\end{subsection}

\begin{subsection}{$\Gamma$-convergence.}\label{rmk:gamma}
The energy $I$ arises as the $\Gamma$-limit of the discrete interaction energies $w_n/n^2$, defined in \eqref{de}, as $n\to\infty$.
This can be proved by following the argument in \cite[Theorem~3.3]{MPS}, where a related energy is derived from a semi-discrete strain energy model.
In particular, the $\Gamma$-liminf inequality follows from a standard lower semicontinuity argument. For the $\Gamma$-limsup inequality
it is enough to consider a measure $\mu$ with $I(\mu)<+\infty$. The argument in Section~\ref{sec:cpt} ensures that one can also assume $\mu$ to have compact support.
In this case, since $V$ satisfies the bounds \eqref{Vbounds}, a recovery sequence can be constructed exactly as in \cite[Theorem~3.3]{MPS}.
\end{subsection}

\end{section}


\begin{section}{Characterisation of the minimiser of $I$: The semi-circle law.}\label{sec:3}
We start by characterising the minimiser of $I$ as the unique measure satisfying the Euler-Lagrange conditions for $I$. Then we will show that the semi-circle law satisfies such conditions.

\begin{subsection}{Euler-Lagrange conditions}
We now derive the Euler-Lagrange conditions for the functional $I$, and show that they characterise the minimiser. This procedure  is the same as for the logarithmic potential (see, e.g., \cite[Theorem~1.3]{SaTo}; see also \cite{CCP, SST}).

We first introduce the notion of capacity. For any compact set $K\subset \R^2$
we define the {\em capacity} of $K$ as
$$
\text{cap}(K):= \Phi \left(\inf_{\mu\in \mathcal{P}(K)} \iint_{\R^2\times\R^2} V(x-y)\,d\mu(x)\, d\mu(y) \right), \qquad \Phi(t) = e^{-t},
$$
where  $\mathcal{P}(K)$ is the class of all probability measures
with support in $K$. For any Borel set $B\subset \R^2$ the capacity of $B$ is defined as the supremum of the capacity of compact sets $K\subset B$. Finally, any set (not necessarily Borel) contained in a Borel set of zero capacity, is
considered to have zero capacity.

In the following we say that a property holds {\em quasi everywhere} (q.e.) in a set $A$
if the set of points in $A$ where the property is not satisfied has zero capacity.
Note that if $B$ is a Borel set with zero capacity and $\mu\in\mathcal{P}(\R^2)$
is a measure with compact support and finite interaction energy, then $\mu(B)=0$. In other words,
any measure with compact support and finite interaction energy does not charge sets of zero capacity.

We also note that the capacity is monotone increasing with respect to inclusion. Moreover, a countable union of sets with zero capacity has zero capacity.

\begin{theorem}\label{mu1-ch}
The minimiser $\mu\in \mathcal{P}(\R^2)$ of $I$ is uniquely characterised by the Euler-Lagrange conditions:
there exists $c\in\R$ such that
\begin{align}
&(V\ast \mu)(x) + \frac{|x|^2}2 = c \qquad \text{for }\mu\text{-a.e.\ }x\in \supp \mu,
\label{EL-2}
\\
&(V\ast \mu)(x) + \frac{|x|^2}2 \geq c \qquad \text{for q.e.\ }x\in \R^2.
\label{EL-1}
\end{align}
\end{theorem}

\begin{remark}
From condition \eqref{EL-2} it follows that the constant $c$ is given by 
$$
c= I(\mu)- \frac12 \int_{\R^2} |x|^2\, d\mu(x).
$$
\end{remark}

\begin{proof}[Proof of Theorem~\ref{mu1-ch}] We divide the proof into two steps: The derivation of the conditions \eqref{EL-2}--\eqref{EL-1}, and the proof of the fact that they characterise the minimiser $\mu$.\smallskip

\noindent
\textit{Step~1: Derivation of the Euler-Lagrange conditions.} 
We consider variations of the minimiser $\mu$ of $I$ of the following form: $(1-\e)\mu + \e\nu$, where $\e\in (0,1)$
and $\nu\in \mathcal{P}(\R^2)$ has compact support and satisfies $I(\nu)<+\infty$. The minimality of $\mu$ implies  
\begin{equation*}
I((1-\e)\mu + \e\nu) \geq I(\mu),
\end{equation*}
which we can rewrite more explicitly as 
$$
\e\left(\int_{\R^2} (2(V\ast \mu) + |x|^2) \,d\nu - (2(V\ast \mu) + |x|^2) \,d\mu\right) + \e^2 \int_{\R^2}(V\ast (\mu-\nu)) \,d(\mu -\nu) \geq 0.
$$
Since the coefficient of the $\e^2$ term is finite, we can divide the previous relation by $\e>0$ and let $\e\to 0^+$ to obtain 
\begin{equation}\label{int-EL}
\int_{\R^2} \left((V\ast \mu)(x) + \frac{ |x|^2}2\right ) \,d\nu(x) \geq  \int_{\R^2}\left((V\ast \mu)(x) + \frac{|x|^2}2\right) \,d\mu(x) =: c,
\end{equation}
which has to be true for every $\nu\in \mathcal{P}(\R^2)$ with compact support and such that $I(\nu)<+\infty$. 

Condition \eqref{int-EL} implies \eqref{EL-1}. Indeed, set $F(x):=(V\ast \mu)(x) + \frac{|x|^2}2$ and assume for contradiction that the set
$\{ x\in\R^2: F(x)<c\}$
has positive capacity. Then there exists $n_0\in\N$ large enough, so that the compact set
$$
K:=\Big\{x\in\R^2: |x|\leq n_0, \ F(x)\leq c-\frac1{n_0}\Big\}
$$
has positive capacity (note that $F$ is lower semicontinuous, which implies that $K$ is closed and thus compact). On the other hand, since
$$
\int_{\R^2} F(x)\, d\mu(x)=c,
$$
there must exist a Borel set $E$, disjoint from $K$, such that $\mu(E)>0$ and
$F(x)>c-\frac1{2n_0}$ for $\mu$-a.e.\ $x\in E$. 
Since $K$ has positive capacity, we deduce from the definition of capacity
that there exists $\tilde\nu\in \mathcal{P}(K)$ with finite interaction energy.
We now consider the measure $\nu\in\mathcal{P}(\R^2)$ defined by
$$
\nu:=\mu+\e\mu(E)\,\tilde\nu -\e\mu\mres E,
$$
where $\e>0$ is sufficiently small. We compute
\begin{eqnarray*}
\int_{\R^2} F(x)\, d\nu(x) & = &\int_{\R^2} F(x)\, d\mu(x)+\e \mu(E)\int_{\R^2} F(x)\, d\tilde\nu(x)
-\e \int_E F(x)\, d\mu(x)
\\
& \leq & c-\frac{\e \mu(E)}{2n_0},
\end{eqnarray*}
which contradicts \eqref{int-EL}. Therefore, \eqref{EL-1} is proved.

Since the minimiser $\mu$ does not charge sets of zero capacity, by \eqref{EL-1} we also have 
$$
(V\ast \mu)(x) + \frac{|x|^2}2 \geq c \qquad \text{for } \mu\text{-a.e.\ }x\in \R^2.
$$
Since
$$
\int_{\R^2}(V\ast \mu)(x)\,d\mu(x) + \frac12\int_{\R^2}|x|^2 \,d\mu(x) = c,
$$
the inequality above implies \eqref{EL-2}.\smallskip

\noindent
\textit{Step~2: Characterisation of the minimiser $\mu$.} Assume that $\tilde \mu \in \mathcal{P}(\R^2)$ with compact support satisfies \eqref{EL-2}--\eqref{EL-1} for some constant $\tilde c$, and define $\mu_t:=t\mu + (1-t)\tilde\mu$ for $t\in (0,1)$. Then 
\begin{eqnarray*}
I(\mu_t) &= & t \int_{\R^2} \left(V\ast \mu + \frac{ |x|^2}2\right ) \,d\mu_t + (1-t) \int_{\R^2} \left(V\ast \tilde\mu + \frac{ |x|^2}2\right ) \,d\mu_t + \frac12\int_{\R^2} |x|^2 \,d\mu_t\\
&\geq & tc + (1-t) \tilde c  + \frac12\int_{\R^2} |x|^2 \big(t \,d\mu + (1-t)\,d \tilde\mu\big) \\
&=& t\left(c + \frac12\int_{\R^2} |x|^2 \,d\mu\right) + (1-t)\left(\tilde c + \frac12\int_{\R^2} |x|^2 \,d\tilde \mu\right) 
= tI(\mu) + (1-t)I(\tilde\mu),
\end{eqnarray*}
which, by the strict convexity of $I$, implies that $\mu=\tilde\mu$.
\end{proof}

\end{subsection}

\begin{subsection}{The semi-circle law}
We now prove the main result of the paper, Theorem~\ref{thm:chara}, that is, we
show that the semi-circle law satisfies the Euler-Lagrange conditions derived in the previous section,
and thus is the unique minimiser of $I$.

\begin{proof}[Proof of Theorem~\ref{thm:chara}]
Set 
$$
F(x):=(V\ast \semi)(x) + \frac12|x|^2 \qquad \text{for every } x\in\R^2.
$$
First of all, we note that the interaction term is given by  
\begin{equation}\label{C-ast}
(V\ast \semi)(x) =\frac1{\pi} \int_{-\sqrt2}^{\sqrt{2}} \left(-\frac12 \log(x_1^2 + (x_2-y_2)^2 ) + \frac{x_1^2}{x_1^2 + (x_2-y_2)^2} \right)\sqrt{2-y_2^2} \, dy_2.
\end{equation}
Therefore, $V\ast \semi$ is the image of a continuous function through a weakly singular integral operator. We can conclude that $F$ is continuous over $\R^2$ and that it is $C^1$ on $\R^2\setminus (\supp\semi)$.

We split the proof into two steps. In the first step we investigate the behaviour of $F$ on the $x_2$-axis, in the second step we show that the Euler-Lagrange conditions are satisfied on the whole of~$\mathbb{R}^2$.\smallskip

\noindent
\textit{Step~1: Behaviour on the $x_2$-axis.}
This is classical, since it corresponds to the fact that the semi-circle law is a minimiser for the logarithmic potential in one dimension. For the sake of completeness, we repeat the arguments here.

On the $x_2$-axis, we have
\begin{equation}\label{C-ast-0}
(V\ast \semi)(0,x_2) =\frac1{\pi} \int_{-\sqrt2}^{\sqrt{2}} \left(-\frac12 \log((x_2-y_2)^2 ) \right)\sqrt{2-y_2^2} \, dy_2.
\end{equation}
We compute its derivative with respect to $x_2$, obtaining 
\begin{equation}\label{Hilbert}
\partial_{x_2} (V\ast \semi)(0,x_2) =
 \begin{cases}
-x_2 - \sqrt{x_2^2-2} \quad &\text{if } x_2< -\sqrt 2,\\
-x_2 \quad &\text{if }  x_2 \in [-\sqrt 2, \sqrt 2],\\
-x_2 + \sqrt{x_2^2-2} \quad &\text{if } x_2 >\sqrt 2.
\end{cases}
\end{equation}
In fact, for $|x_2|>\sqrt{2}$, it can be easily seen that
$$\partial_{x_2} (V\ast \semi)(0,x_2)=
 - \frac{1}{\pi} \int_{-\sqrt2}^{\sqrt{2}} \frac{\sqrt{2-y_2^2}}{x_2-y_2} \, dy_2.$$
 Such a formula remains true for $x_2\in (-\sqrt{2},\sqrt{2})$ in the sense of distributions, where the integral at the right-hand side is to be intended as the principal value. Therefore, for any $|x_2|\neq\sqrt{2}$, $\partial_{x_2} (V\ast \semi)(0,x_2)$ coincides with the Hilbert transform of $\chi_{(-\sqrt{2},\sqrt{2})}(y_2)\sqrt{2-y_2^2}$. This can be computed explicitly (see, e.g., \cite[Chapter~4]{Meh}) leading to formula \eqref{Hilbert} for any $|x_2|\neq \sqrt{2}$. By passing to the limit, we easily conclude that \eqref{Hilbert} holds in the classical sense for every $x_2\in\R$.
We immediately infer that, for some constant $c_1$, we have
\begin{equation}\label{ELsupp}
F(0,x_2)=(V\ast \semi)(0,x_2) + \frac{x_2^2}2=c_1\qquad \text{for every }x_2\in[-\sqrt{2},\sqrt{2}]
\end{equation}
and 
\begin{equation}\label{ELx2axis}
F(0,x_2)=(V\ast \semi)(0,x_2) + \frac{x_2^2}2>c_1\qquad  \text{for every }x_2\in\R \setminus [-\sqrt{2},\sqrt{2}].
\end{equation}

We observe that 
$$
c_1=(V\ast \semi)(0)=
I(\semi)- \frac12 \int_{\R^2} |x|^2\, d\semi(x) ,
$$
and we now compute $c_1$ and thus the energy of the semi-circle law.
We have
\begin{equation}\label{p:evaluation2}
c_1 = (V\ast \semi)(0) = \frac1{\pi} \int_{-\sqrt2}^{\sqrt{2}} \left(- \log|y_2| \right)\sqrt{2-y_2^2} \, dy_2 = \frac12+\frac12\log2.
\end{equation}
For the confinement term we have
\begin{equation}\label{conf_sc}
\int_{\R^2}|x|^2 \,d\semi = \frac{1}{\pi}\int_{-\sqrt2}^{\sqrt2} x_2^2 \sqrt{2-x_2^2} \,dx_2 = \frac12.
\end{equation}
Therefore we conclude that 
$$
I(\semi)=\frac34 +\frac12\log2\ (\approx 1.0966).
$$
\smallskip

\noindent
\textit{Step~2: Euler-Lagrange conditions.} We now show that $\semi$ satisfies \eqref{EL-2-intro}--\eqref{EL-1-intro}. 
Note that \eqref{EL-2-intro} follows immediately from \eqref{ELsupp} and \eqref{p:evaluation2}. 

It remains to show that $\semi$ satisfies the Euler-Lagrange condition \eqref{EL-1-intro}. We claim that $F(x) > \frac12 +\frac12\log2$ for every $x\in\R^2$ such that $x_1\neq 0$. Then, by \eqref{ELsupp} and \eqref{ELx2axis}, the proof would be concluded.
Since $F$ is even in $x_1$ and $x_2$, and again by \eqref{ELsupp} and \eqref{ELx2axis}, the claim follows if we show that
\begin{equation}\label{claim-pos}
\partial_{x_1} F(x_1,x_2) > 0 \qquad \text{for every } x_1> 0,\ x_2\geq 0.
\end{equation}

In the following we consider $x_1> 0$ and $x_2\geq 0$.
We observe that 
\begin{equation}\label{primitive:V}
V(x_1,x_2) = - \partial_{x_2}\left(-x_2+\frac12 x_2 \log(x_1^2+x_2^2)\right).
\end{equation}
By rewriting \eqref{C-ast}, using also \eqref{primitive:V}, we have
\begin{eqnarray*}
(V\ast \semi)(x) &= & \frac1{2\pi} \int_{-\sqrt2}^{\sqrt{2}} \int_{-\sqrt{2-y_2^2}}^{\sqrt{2-y_2^2}} \left(-\frac12 \log(x_1^2 + (x_2-y_2)^2 ) + \frac{x_1^2}{x_1^2 + (x_2-y_2)^2} \right) \,dy_1 \,dy_2\\
&= & \frac1{2\pi} \int_{-\sqrt2}^{\sqrt{2}} \int_{-\sqrt{2-y_1^2}}^{\sqrt{2-y_1^2}} \left(-\frac12 \log(x_1^2 + (x_2-y_2)^2 ) + \frac{x_1^2}{x_1^2 + (x_2-y_2)^2} \right) \,dy_2 \,dy_1\\
& = & 1+ \frac1{4\pi} \int_{-\sqrt2}^{\sqrt{2}} \Big(x_2 - \sqrt{2-y_1^2}\Big) \log\Big(x_1^2 + \Big(x_2- \sqrt{2-y_1^2}\Big)^2 \Big)\,dy_1\\
& & {}-\frac1{4\pi} \int_{-\sqrt2}^{\sqrt{2}} \Big(x_2 + \sqrt{2-y_1^2}\Big) \log\Big(x_1^2 + \Big(x_2+ \sqrt{2-y_1^2}\Big)^2 \Big)\,dy_1.
\end{eqnarray*}
Hence we have that 
\begin{eqnarray*}
\partial_{x_1} F(x_1,x_2) &= & x_1 + \frac1{4\pi} \int_{-\sqrt2}^{\sqrt{2}} \Big(x_2 - \sqrt{2-y_1^2}\Big) \frac{2x_1}{x_1^2 + \Big(x_2- \sqrt{2-y_1^2}\Big)^2}\,dy_1\\
& & {}-\frac1{4\pi} \int_{-\sqrt2}^{\sqrt{2}} \Big(x_2 + \sqrt{2-y_1^2}\Big)\frac{2x_1}{x_1^2 + \Big(x_2+ \sqrt{2-y_1^2}\Big)^2}\,dy_1.
\end{eqnarray*}
We now change variables, setting $y_1 = \sqrt 2 \sin\theta$, for $\theta \in \left[-\frac\pi2,\frac\pi2\right]$; setting also $\xi_i:= \frac{x_i}{\sqrt 2}$, we have 
\begin{eqnarray*}
\partial_{x_1} F(\sqrt{2}\xi_1,\sqrt{2}\xi_2) &= & \sqrt 2\xi_1 + \frac{\sqrt 2}{2\pi} \int_{-\frac\pi2}^{\frac\pi2} \frac{\xi_1(\xi_2 - \cos\theta)\cos\theta}{\xi_1^2 + (\xi_2- \cos\theta)^2}\,d\theta
 - \frac{\sqrt 2}{2\pi} \int_{-\frac\pi2}^{\frac\pi2} \frac{\xi_1(\xi_2 + \cos\theta)\cos\theta}{\xi_1^2 + (\xi_2+ \cos\theta)^2}\,d\theta\\
 &= &\sqrt 2\xi_1 + \frac{\sqrt 2}{2\pi} \int_{-\frac\pi2}^{\frac\pi2} \frac{\xi_1(\xi_2 - \cos\theta)\cos\theta}{\xi_1^2 + (\xi_2- \cos\theta)^2}\,d\theta
 + \frac{\sqrt 2}{2\pi} \int_{\frac\pi2}^{\frac{3\pi}2} \frac{\xi_1(\xi_2 - \cos\hat\theta)\cos\hat\theta}{\xi_1^2 + (\xi_2- \cos\hat\theta)^2}\,d\hat\theta\\
 & = &\sqrt 2\xi_1 + \frac{\sqrt 2}{2\pi} \int_{-\pi}^{\pi} \frac{\xi_1(\xi_2 - \cos\theta)\cos\theta}{\xi_1^2 + (\xi_2- \cos\theta)^2}\,d\theta,
\end{eqnarray*}
where we have used the substitution $\hat\theta=\theta+\pi$, and then the periodicity of the cosine function.

For what follows it is convenient to manipulate the expression above slightly, as 
\begin{equation*}
\partial_{x_1} F(\sqrt{2}\xi_1,\sqrt{2}\xi_2) = \frac{\sqrt 2}{2\pi}\xi_1^2 \int_{-\pi}^{\pi} \frac{\xi_1}{\xi_1^2 + (\xi_2- \cos\theta)^2}\,d\theta + 
\frac{\sqrt 2}{2\pi}\xi_1\xi_2 \int_{-\pi}^{\pi} \frac{\xi_2 - \cos\theta}{\xi_1^2 + (\xi_2- \cos\theta)^2}\,d\theta.
\end{equation*}
In terms of these new variables, the claim \eqref{claim-pos} corresponds to proving that 
\begin{equation*}
\xi_1^2 \int_{-\pi}^{\pi} \frac{\xi_1}{\xi_1^2 + (\xi_2- \cos\theta)^2}\,d\theta + 
\xi_1\xi_2 \int_{-\pi}^{\pi} \frac{\xi_2 - \cos\theta}{\xi_1^2 + (\xi_2- \cos\theta)^2}\,d\theta > 0 
\end{equation*}
for every $\xi_1> 0$, $\xi_2\geq 0$.
Clearly this is true for  
$\xi_2=0$, since the expression in this case reduces to 
$$
\xi_1^2 \int_{-\pi}^{\pi} \frac{\xi_1}{\xi_1^2 + (\cos\theta)^2}\,d\theta,
$$
which is positive for $\xi_1> 0$. Hence we only need to show that 
\begin{equation}\label{claim-pos_1}
\xi_1 \int_{-\pi}^{\pi} \frac{\xi_1}{\xi_1^2 + (\xi_2- \cos\theta)^2}\,d\theta + 
\xi_2 \int_{-\pi}^{\pi} \frac{\xi_2 - \cos\theta}{\xi_1^2 + (\xi_2- \cos\theta)^2}\,d\theta > 0
\end{equation}
for every $\xi_1> 0$, $\xi_2>0$.
To evaluate the integrals above, we will make use of the following identity:
\begin{equation}\label{identity:SaTo}
\frac1{2\pi} \int_{-\pi}^{\pi}\log |z-\cos\theta| \,d\theta =  \log|z+\sqrt{z^2-1}|- \log 2,
\end{equation}
where $z\in \mathbb{C}\setminus [-1,1]$ and $\sqrt{z^2-1}$ here and in what follows denotes the branch of the complex square root that behaves asymptotically as $z$ at infinity. Namely, for $z\in\mathbb{C}\setminus [-1,1]$ such that
$z=\rho e^{i\theta}$ with $\rho>0$ and $0\leq \theta<\pi$, we have $z^2-1=\rho_1e^{i\theta_1}$ with $\rho_1>0$ and $0\leq \theta_1<2\pi$ and $\sqrt{z^2-1}=\sqrt{\rho_1}e^{i\theta_1/2}$. Instead, if $z\in\mathbb{C}\setminus [-1,1]$ is such that
$z=\rho e^{i\theta}$ with $\rho>0$ and $\pi\leq \theta<2\pi$, we have $z^2-1=\rho_1e^{i\theta_1}$ with $\rho_1>0$ and $2\pi\leq \theta_1<4\pi$ and $\sqrt{z^2-1}=\sqrt{\rho_1}e^{i\theta_1/2}$. For the proof of identity \eqref{identity:SaTo} we refer to 
\cite[Example~1.3.5]{SaTo}, where the integral in \eqref{identity:SaTo}  is computed by applying the Joukowsky transformation.
 
Writing explicitly $z=\eta_1+i\eta_2$, and assuming that $\eta_1>0$ and $\eta_2>0$,
the left-hand side of \eqref{identity:SaTo} becomes 
\begin{equation}\label{identity:SaTo_1}
\frac1{2\pi} \int_{-\pi}^{\pi}\log |z-\cos\theta| \,d\theta=\frac1{4\pi} \int_{-\pi}^{\pi}\log \big((\eta_1-\cos\theta)^2 + \eta_2^2\big) \,d\theta=: g(\eta_1,\eta_2).
\end{equation}
From now on we write $g(z)$ or $g(\eta_1,\eta_2)$ to denote the function in \eqref{identity:SaTo}--\eqref{identity:SaTo_1}.

For the derivatives of $g$ with respect to $\eta_1$ and $\eta_2$ we have 
\begin{equation}\label{identity:SaTo_2}
\begin{array}{c}
\displaystyle \partial_{\eta_1} g (\eta_1,\eta_2)= \frac{1}{2\pi} \int_{-\pi}^{\pi} \frac{\eta_1-\cos\theta}{(\eta_1-\cos\theta)^2 + \eta_2^2}\,d\theta,\bigskip\\
\displaystyle \partial_{\eta_2} g (\eta_1,\eta_2)= \frac{1}{2\pi} \int_{-\pi}^{\pi} \frac{\eta_2}{(\eta_1-\cos\theta)^2 + \eta_2^2}\,d\theta,
\end{array}
\end{equation}
which correspond to the integrals in \eqref{claim-pos_1}, provided we pick $\eta_1=\xi_2$ and $\eta_2=\xi_1$. 

Using \eqref{identity:SaTo_2} we find that the claim \eqref{claim-pos_1} is equivalent to proving 
\begin{equation}\label{inturn}
\eta_1 \partial_{\eta_1} g(\eta_1,\eta_2) + \eta_2 \partial_{\eta_2} g(\eta_1,\eta_2)  > 0 \qquad \text{for every }\eta_1>0,\ \eta_2>0.
\end{equation}
Since $g$ is real-valued, \eqref{inturn} is in turn equivalent to 
\begin{equation}\label{claim:complex}
\Re\left(z\,\partial_z g(z)\right)  > 0 \qquad \text{for every }\eta_1>0,\ \eta_2>0,
\end{equation}
where $\partial_z = \frac12\partial_{\eta_1} -\frac{i}{2} \partial_{\eta_2}$, while
$\partial_{\bar z} = \frac12\partial_{\eta_1} +\frac{i}{2} \partial_{\eta_2}$.

To compute the complex derivative of $g$ it is convenient to rewrite $g$ as 
\begin{equation}\label{e:g}
g(z) = \frac12\log(h(z) \overline{h(z)})- \log 2, \qquad h(z):= z+\sqrt{z^2-1}.
\end{equation}
We note that, for $\eta_1,\eta_2>0$, the branch of the square root appearing in the definition of $h$ satisfies 
$\Re (\sqrt{z^2-1})>0$, $\Im (\sqrt{z^2-1})>0$, thus in particular we have $\Re( h(z))>0$, $\Im( h(z))>0$.

Moreover, $h$ is holomorphic, and therefore $\partial_{\bar z}h=0$, which implies that 
\begin{equation}\label{dzh}
\partial_z \bar h = \overline{\partial_{\bar z} h} =0. 
\end{equation}
On the other hand
\begin{equation}\label{dzh_1}
\partial_z h (z)= 1  +\frac{z}{\sqrt{z^2-1}} = \frac{h(z)}{\sqrt{z^2-1}} =  \frac{h(z) \overline{\sqrt{z^2-1}}}{|z^2-1|^2}.
\end{equation}
Using \eqref{dzh} and \eqref{dzh_1}, we can then compute the complex derivative of $g$:
$$
\partial_z g (z)= \frac12\frac{1}{|h(z)|^2} \partial_z(h(z) \overline{h(z)}) =  \frac12\frac{1}{|h(z)|^2} \left(\partial_zh(z)\,\overline{h(z)} + h(z)\, \partial_z\overline{h(z)} \right)
 = \frac12 \frac{\overline{\sqrt{z^2-1}}}{|z^2-1|^2}.
$$
Finally, we can compute the left-hand side of \eqref{claim:complex} for $z=\eta_1+i\eta_2$ with $\eta_1,\eta_2>0$:
\begin{eqnarray*}
\Re\left(z\, \partial_z g(z)\right) &= &\frac12 \frac1{{|z^2-1|^2}} \Re\left(z\, \overline{\sqrt{z^2-1}}\right) \\
& =  &\frac12 \frac1{{|z^2-1|^2}} \left(\eta_1\Re( \sqrt{z^2-1}) + \eta_2 \Im( \sqrt{z^2-1})\right) > 0,
\end{eqnarray*}
since, as said above, $\Re (\sqrt{z^2-1})>0$, $\Im (\sqrt{z^2-1})>0$ when $\eta_1,\eta_2>0$.

This proves \eqref{claim:complex}, and hence \eqref{claim-pos}, which implies that the measure $\semi$ satisfies the Euler-Lagrange condition \eqref{EL-1-intro} and concludes the proof of the theorem.
\end{proof}

\end{subsection}
\end{section}

\bigskip\bigskip

\noindent
\textbf{Acknowledgements.}
The first two authors are partly supported by GNAMPA--INdAM. MGM acknowledges support 
by the ERC under Grant No.\ 290888
``Quasistatic and Dynamic Evolution Problems in Plasticity and Fracture''.
LR acknowledges support by the Universit\`a di Trieste through FRA~2014.
LS acknowledges support by the EPSRC under the Grant EP/N035631/1 ``Dislocation patterns 
beyond optimality''.

\end{document}